\documentclass[12pt]{article}
\usepackage{openwork}
\usepackage{amsmath, amssymb,bm}
\usepackage{amsmath, amsthm, amscd, amsfonts, amssymb, graphicx, color}

\newtheorem{theorem}{Theorem}[section]
\newtheorem{lemma}[theorem]{Lemma}

\theoremstyle{definition}
\newtheorem{definition}[theorem]{Definition}
\newtheorem{example}[theorem]{Example}

\theoremstyle{remark}

\numberwithin{equation}{section}

\def\E{{\mathbb {E}}}                       
                       
                       \def\R{{\mathbb {R}}}
\def\I{{\mathbb {I}}}                       
\def\Eh{{\hat{\mathbb{E}}}}
\def\Vh{{\hat{\mathbb{V}}}}

\def\PD{{\mathcal P}}                     
                     \def\ED{{\cal E}}
\def\FD{{\mathcal F}}

\def\HD{{\mathcal H}}

\def\ED{{\mathcal E}}

\def\tsc#1{\csdef{#1}{\textsc{\lowercase{#1}}\xspace}}
\tsc{WGM}
\tsc{QE}
\tsc{EP}
\tsc{PMS}
\tsc{BEC}
\tsc{DE}
\title{Strong law of large numbers for $\varphi$-sub-Gaussian random variables under sub-linear expectation spaces}
\author[1,*]{Nyanga Honda Masasila}
\author[2]{Istv\'an Fazekas}
\affil[1]{Doctoral School of Informatics, University of Debrecen, Hungary}
\affil[2]{Faculty of Informatics, University of Debrecen, Hungary}

\date{}

\DeclareFieldFormat{labelnameyear}{}
\DeclareFieldFormat{labelyear}{}
\DeclareFieldFormat{extradate}{}
\begin{document}

\maketitle

\begin{abstract}
We introduce the notions of sub-Gaussian and $\varphi$-sub-Gaussian random variables in sub-linear expectation spaces.
To avoid the problem caused by the existence of two different expectations, i.e., 
the upper expectation and the lower expectation,
we divide the definition of the sub-Gaussian property into an upper part and a lower part.
It turns out that this approach fits well to the sub-linear setting; it provides a proper framework for extending Zajkowski's general result (\cite{Zajkowski2021}) to sublinear expectation spaces. 
Within our framework, we establish a strong law of large numbers for sub-Gaussian sequences. 
We present an example showing the usefulness of our results. \\

\end{abstract}

\section{Introduction}

 In the usual probability framework, the strong law of large numbers (SLLN) relies on linearity of expectation and additivity of probability measure. In many modern applications, including risk measures, robust finance, and models with ambiguity, these assumptions may fail.
This has motivated the development of \emph{sub-linear expectation spaces}, 
in which expectations are sub-linear functionals and probabilities (capacities) are sub-additive.
In this setting, the SLLN takes another form: 
it asserts that every cluster point of the sequence of empirical averages lies between the lower and upper expectations,
 with lower capacity equal to 1. 
 Formally,
\begin{equation}
\label{slln}
  v\!\left(\omega\in \Omega :\ED(X_1) \le\liminf_{n \to \infty} \frac{1}{n} \sum_{k=1}^n X_k(\omega) \le 
  \limsup_{n \to \infty} \frac{1}{n} \sum_{k=1}^n X_k(\omega) \le \Eh(X_1) \right) = 1,  
\end{equation}
where $\Eh$ denotes the upper (sub-linear) expectation, $\ED$ the lower (conjugate) expectation, 
and $v$ the corresponding lower capacity. 

Building on this framework, numerous authors have investigated versions of the strong law of large numbers (SLLN) under various sets of assumptions. In particular, \cite{marinacci1999limit} and \cite{MaccheroniMarinacci2005} established an SLLN for bounded and continuous random variables under a totally monotone capacity. Subsequently, \cite{ChenZ2013} proved an SLLN for independent random variables assuming finite $(1+\alpha)^{th}$ moments of the upper expectation, while \cite{Chen2012} further relaxed the framework by removing the continuity assumption on the upper capacity.

Sub-Gaussian distributions were studied  in \cite{Kahane1960}. Then 
it was used to prove laws of large numbers (see, e.g., \cite{Chow1966},  \cite{Taylor1987} and \cite{Buldygin2000}); see also \cite{Antonini2013} for related results. Recently, \cite{Zajkowski2021} developed a general strong law for $\varphi$-sub-Gaussian sequences under linear expectations. The purpose of the present paper is to extend this framework to sub-linear expectation spaces.
\section{Basic Properties of Capacities and Sub-linear Expectations} \label{sublinear}
%%%%%%%%%%%%%%%%%%%%%%%%%%%%%%%%%%%%%%%%beginFIFI
Based on the classical results of  \cite{Choquet1954} and subsequent developments by \cite{Huber1973}, 
we summarize the properties of capacities.
Let $\Omega$ be a non-empty set and let $\FD$ be a $\sigma$-algebra of subsets of $\Omega$. 

\begin{definition} \label{defV}
$\Vh$ is called a sub-additive probability (upper probability or upper capacity), if 

\leavevmode
\begin{enumerate}
    \item (Normalized.) $\Vh(\Omega) = 1$, $\Vh(\emptyset) = 0$.
    \item (Monotone.) If $A, B \in \FD$ and $A \subset B$, then $\Vh(A) \leq \Vh(B)$.
    \item (Sub-additive.) If $A, B \in \FD$, then
    \[
    \Vh\left(A \cup B \right) \leq \Vh(A) + \Vh(B).
    \]
    \item (Lower-continuous.) If $A_n \uparrow A$, $A_n \in \FD$, 
    then $\Vh(A_n) \uparrow \Vh(A)$.
\end{enumerate}
\end{definition}
Sub-additivity and lower-continuity imply $\sigma$-sub-additivity:
  For any sequence $\{A_n\}_{n=1}^\infty \subset \FD$,
    \[
    \Vh\left(\bigcup_{n=1}^\infty A_n\right) \leq \sum_{n=1}^\infty \Vh(A_n).
    \]
    
An event $A$ is called a quasi-sure (q.s.) event if $\Vh(A^c)=0$, 
where $A^{c}$ denotes the complement of the event $A$ in $\Omega$.

The lower capacity corresponding to $\Vh$ can be expressed in terms of $\Vh$ as $v(A) = 1 - \Vh(A^{c})$.
The set functions $\Vh$ and $v$ form a pair of conjugate capacities.

\begin{example}  \label{exmplV}
The basic method for obtaining a sub-additive probability is as follows.
Let $\mathcal{P}$ be a family of usual (that is, $\sigma$-additive) probabilities on $(\Omega, \FD)$. 
Introduce $\Vh$ as
\begin{equation}  \label{Vsup}
\Vh(A) = \sup_{Q \in \PD} Q(A), \quad \forall A \in \FD.
\end{equation}
It is easy to see that $\Vh$ satisfies the properties listed in Definition \ref{defV}.
The corresponding lower capacity is 
\begin{equation*}  
v(A) = \inf_{Q \in \PD} Q(A), \quad \forall A \in \FD.
\end{equation*}
\end{example}
Standard results such as Chebyshev’s inequality and the Borel–Cantelli lemma remain valid in this framework (see, \cite{ChenZ2013} and \cite{Peng2019}).

%%%%%%%%%%%%%%%%%%%%%%%%%%%%%%%%%%%%%%%%%%%%%%%%%%%%%%%%%%%%%
Now, we list basic properties of the sub-linear expectation $\Eh$; see \cite{Peng2019} for details.
We mention that, in certain papers, $\Eh$ is called the upper expectation or sub-additive expectation.
As usual, an extended real-valued function $X$ on $\Omega$ is called a random variable if
$X^{-1}(A) \in \FD$ for any Borel set $A$.
We assume, that there exists a subset $\HD$ of the random variables and
an extended real-valued function $\Eh[X]$ of $X\in \HD$ 
so that the assumptions of Definition \ref{defE} are satisfied.
We shall suppose that the non-negative random variables belong to $\HD$.
In Definition \ref{defE}, the case $ \Eh[X] + \Eh[Y]= \infty + (-\infty)$ is excluded.

%%%%%%%%%%%
\begin{definition} \label{defE}
An extended real-valued function $\Eh[X]$ of $X\in \HD$ is called a sub-linear expectation if
it satisfies the following properties.
\begin{enumerate}
    \item Monotone: If $X \leq Y$, then $\Eh[X] \leq \Eh[Y]$.
    \item Constant preserving: $\Eh[c] = c$, for any $c \in \R$.
    \item Sub-additive: $\Eh[X+Y] \leq \Eh[X] + \Eh[Y]$.
    \item Positive homogeneous: $\Eh[\lambda X] = \lambda \Eh[X]$ for any constant $\lambda \geq 0$.
    \item Monotone convergence: If $X_n \uparrow X$, and $X_1\ge 0$,  then $\Eh[X_n] \uparrow \Eh[X]$.
\end{enumerate}
\end{definition}

$(\Omega, \HD, \Eh)$ is called a sub-linear expectation space.

If the sub-linear expectation $\Eh[\,.\,]$ is given in advance, then 
we can introduce a sub-additive probability $\Vh$
by $\Vh(A) = \Eh [\I_A]$, for any event $A$,  where $\I_A$ denotes the indicator of $A$.
Then this $\Vh$ satisfies the properties given in Definition \ref{defV}.

Throughout this paper, we assume that an upper probability $\Vh$ is given on $(\Omega, \FD)$, 
and a sub-linear expectation is given on $\HD$ so that $\Vh(A) = \Eh [\I_A]$.
Relations of random variables are considered as quasi-sure relations, e.g.,
$\Eh[X]= \Eh[Y]$ if $X=Y$ q.s.

%%%%%%%%%%%%
\begin{example} \label{ExmplE}
The usual method to define a sub-additive expectation is as follows.
Let $\PD$ be a family of probabilities on $(\Omega, \FD)$. 
Let
\begin{equation}  \label{ESupE}
\Eh[X] = \sup_{Q \in \PD} \E_Q[X], 
\end{equation}
where $\E_Q$ is the usual expectation corresponding to $Q \in \PD$ such that
%\[
$\E_Q[\I_A] = Q(A), \quad \forall A \in \FD$.
%\]
Then, properties listed in Definition \ref{defE} are satisfied.
\end{example}

We remark that if an abstract sub-linear expectation $\Eh [\,.\,]$ is given with the properties listed in Definition \ref{defE}, 
and certain additional conditions are satisfied, then $\Eh$ has a representation \eqref{ESupE}.
For the precise statement, see \cite{Delbaen2002} and \cite{Peng2010}.

For $X \in \HD$, $\Eh(X)$ can be called supermean, whereas $\ED(X) = -\Eh(-X)$ is called submean. 
By the above properties of $\Eh(X)$, we have $\ED(X) \le \Eh(X)$. 
If $\ED (X) \ne -\Eh(X)$,
then $X$ is said to have mean uncertainty.

%%%%%%%%%%%%%%%%%%%%%%%%%%%%%%%%%%

\section{Sub-Gaussian Random Variables under Sub-linear Expectations}
\label{subgauss}
For a usual probability space $(\Omega,\FD,P)$, a random variable
$X:\Omega\to\R$ is called \emph{sub-Gaussian} if there exists a constant
$a\in[0,\infty)$ such that
\begin{equation}  \label{eq:classical_subgaussian}
\E_P\!\left[e^{\lambda X}\right]
\le
\exp\!\left(\frac{a^2\lambda^2}{2}\right),
\qquad \text{for all} \ \ \lambda\in\R.
\end{equation}

Condition \eqref{eq:classical_subgaussian} shows that a random variable is
sub-Gaussian if and only if its moment generating function is dominated by that
of a centered Gaussian random variable with variance parameter $a^2$. 
The term \emph{sub-Gaussian} reflects this Gaussian-type exponential moment control rather
than any exact distributional identity.
We see that condition \eqref{eq:classical_subgaussian} holds only for zero-mean random variables.
For sub-linear expectation, we shall replace it with two one-sided conditions.
More generally, we shall do it for $\varphi$-sub-Gaussian property.

%%%%%%%%%%%%%%%%%%%%%%%%%%%%%%%%%%%%%%%%%%%%%%%%%%%%%%%%%%%%%%%%%%%%%
%begin FI
%%%%%%%%%%%%%%%%%%%%%%%%%%%%%%%%%%%%%%%%%%%%%%%%%%%%%%%%%%%%%%%%%%%%%%

Therefore, we shall introduce the notion of $\varphi$-sub-Gaussian random variables for sub-linear expectation spaces.
For usual probability spaces, this was studied e.g., in \cite{Zajkowski2021}.
We call a real continuous even convex function $\varphi(x)$, $x \in \R$, an
$N$-function if the following conditions are satisfied:
\begin{itemize}
\item[(a)] $\varphi(0) = 0$ and $\varphi(x)$ is monotone increasing for $x > 0$,
\item[(b)] $\lim_{x \to 0} \frac{\varphi(x)}{x} = 0$ and
$\lim_{x \to \infty} \frac{\varphi(x)}{x} = \infty$.
\end{itemize}

An $N$-function $\varphi$ is called a quadratic $N$-function if, in addition,
$\varphi(x) = c x^2$ for all $|x| \le x_0$ with $c > 0$ and $x_0 > 0$.

In our results, we shall use the following quadratic $N$-function, see \cite{Zajkowski2021}.
\begin{definition} 
For $p \ge 1$, let
\[
\varphi_p(x) =
\begin{cases}
\dfrac{x^2}{2}, &   \text{if} \ \ |x| \le 1, \\[2ex]
\dfrac{1}{p}|x|^p - \dfrac{1}{p} + \dfrac{1}{2}, & \text{if} \ \ |x| > 1.
\end{cases}
\]
\end{definition}
The function $\varphi_p$ can be considered as a standardization of the function $|x|^p$.
We shall see that, for $p = 2$, we have the case of the usual sub-Gaussian random variables.

\begin{definition}
Let $\varphi$ be a quadratic $N$-function.
A random variable $\xi$ is said to be $\varphi$-sub-Gaussian under the sub-linear expectation $\Eh$ if there exist
constants  $a > 0$, $\overline{m}$, and $\underline{m}$, such that
\begin{equation}   \label{upGauss}
\Eh e^{\lambda(\xi - \overline{m}}) \le e^{\varphi(a\lambda)}, \quad {\text {for}} \quad \lambda>0,
\end{equation}
and 
\begin{equation}   \label{lowGauss}
\Eh e^{\lambda(\xi - \underline{m}}) \le e^{\varphi(a\lambda)}, \quad {\text {for}} \quad \lambda<0.
\end{equation}
$a > 0$, $\overline{m}$, and $\underline{m}$ are parameters of the $\varphi$-sub-Gaussian variable.
\end{definition}

For the upper bound, we shall use convex conjugate.
For any real function $\varphi(x)$, $x \in \R$, 
the function $\varphi^*(y)$, $y \in \R$, is called the convex conjugate of $\varphi$, if
$\varphi^*(y) = \sup_{x \in \R} \{ xy - \varphi(x) \}$.

The following properties are known for the convex conjugate, see e.g. \cite{Zajkowski2020}.
If $\varphi$ is a quadratic $N$-function, then $\varphi^*$ is also a quadratic $N$-function.
For $p, q > 1$ such that $\frac{1}{p} + \frac{1}{q} = 1$, we have
$\varphi_p^* = \varphi_q$.
The convex conjugation is order-reversing, i.e.  if $\varphi_1 \ge \varphi_2$, then $\varphi_1^* \le \varphi_2^*$.
The following scaling property holds.
For $a > 0$ and $b \ne 0$, let $\psi(x) = a \varphi(bx)$.
Then
$\psi^*(y) = a \varphi^*\left( \frac{y}{ab} \right)$.

\begin{lemma} \label{mainLemma}
Assume that \eqref{upGauss} and \eqref{lowGauss} are satisfied.
Then for $\varepsilon>0$ we have
\begin{equation}
\Vh\left( \{ \xi-\overline{m} > \varepsilon \} \cup \{ \xi-\underline{m} < -\varepsilon \} \right)  
\le 2 e^{-\varphi^*\left(\varepsilon/a\right)}.
\end{equation}

\end{lemma}

\begin{proof}
Let $\varepsilon>0$ and $\lambda>0$.
Then, by the Chebyshev inequality and the $\varphi$-sub-Gaussian property, we have
\begin{eqnarray}
\Vh(\xi-\overline{m} > \varepsilon) =
\Vh \left( e^{ \lambda(\xi -\overline{m})} > e^{\lambda\varepsilon} \right)  
\le
\frac{1}{e^{\lambda\varepsilon}} \Eh \left( e^{ \lambda(\xi -\overline{m}) }\right)
\le
\frac{1}{e^{\lambda\varepsilon}} e^{ \varphi(a\lambda)}
=e^{-( \lambda\varepsilon- \varphi(a\lambda))}.
\end{eqnarray}
To find the optimal upper bound that is offered by the above inequality, we use convex conjugate.
So we obtain
\begin{equation}
\Vh(\xi-\overline{m} > \varepsilon)  \le e^{-\varphi^*\left(\varepsilon/a\right)}, \quad \text{for} \ \ \varepsilon>0.
\end{equation}

For the other side, let $\varepsilon<0$ and $\lambda<0$.
Then, by the Chebyshev inequality and the $\varphi$-sub-Gaussian property, we have
\begin{eqnarray}
\Vh(\xi-\underline{m} < \varepsilon) \le
\frac{1}{e^{\lambda\varepsilon}} \Eh \left( e^{ \lambda(\xi -\underline{m}) }\right)
\le
e^{-( \lambda\varepsilon- \varphi(a\lambda))}.
\end{eqnarray}
From this inequality, the optimal upper bound is
\begin{equation}
\Vh(\xi-\overline{m} < \varepsilon)  \le e^{-\varphi^*\left(\varepsilon/a\right)}, \quad \text{for} \ \ \varepsilon<0.
\end{equation}
As $\varphi^*$ is an even function, we obtain the result.
\end{proof}

To obtain the optimal value of the parameter $a$, we need the following quantity.
\begin{definition} \label{tau}
For a $\varphi$-sub-Gaussian random variable $\xi$ with fixed $\overline{m}$ and $\underline{m}$
 let $\tau_\varphi(\xi)$ be defined as
\begin{eqnarray*}
\tau_\varphi(\xi)
=
\inf\Big\{
a \ge 0 \, :\, 
\Eh e^{\lambda(\xi - \overline{m}}) \le e^{\varphi(a\lambda)}, \ \ {\text {for}} \ \ \lambda>0, \quad 
\text{and} \ \ \Eh e^{\lambda(\xi - \underline{m}}) \le e^{\varphi(a\lambda)}, \ \ {\text {for}} \ \ \lambda<0  
\Big\}.
\end{eqnarray*}
\end{definition}

\section{The Main Result} \label{main}

The following theorem is an extension of Theorem 2.1 of \cite{Zajkowski2021} 
to sub-linear expectation.
\begin{theorem} \label{MainTheorem}
For some $p>1$, let $Z_n, n\ge1$, be a sequence of  $\varphi_p$-sub-Gaussian random variables with parameters
$\overline{m}$ and $\underline{m}$.
Let $\tau_\varphi(Z_n)$ be defined according to Definition \ref{tau}.
If there exist positive numbers $c$ and $\alpha$ such that for every natural number $n$, 
the condition $\tau_{\varphi_p}\!( Z_n )\le c\, n^{-\alpha}$ is satisfied, then
\begin{equation}
\label{upper}
\Vh\left(
\left\{\liminf_{n\to\infty}{Z_n} < \underline{m}\right\}
\cup
\left\{\limsup_{n\to\infty}{Z_n} > \overline{m}\right\}
\right)
= 0
\end{equation}
and 
\begin{equation}
\label{lower}
v\left(
\underline{m}
\le \liminf_{n\to\infty}{Z_n}
\le \limsup_{n\to\infty}{Z_n}
\le \overline{m}
\right)
= 1.
\end{equation}
\end{theorem}
\begin{proof}
Since $V(\cdot)$ and $v(\cdot)$ are conjugate to each other, \eqref{lower} is equivalent to \eqref{upper}.

By Lemma \ref{mainLemma}, for $\varepsilon>0$ we have
\begin{equation}  \label{V}
\Vh\left( \{ Z_n-\overline{m} > \varepsilon \} \cup \{ Z_n-\underline{m} < -\varepsilon \} \right)  
\le 2 \exp\left(-\varphi_q\left(\frac{\varepsilon}{\tau_{\varphi_p}(Z_n)}\right)\right) ,
\end{equation}
where $1/p+1/q=1$ and we applied that  $\varphi_q=\varphi_p^{*}$.
By the condition $\tau_{\varphi_p}\!( Z_n )\le c\, n^{-\alpha}$, 
for all sufficiently large $n$ we have $\frac{\varepsilon}{\tau_{\varphi_p}(Z_n)}>1$, so using the definition of
$\varphi_q$, we obtain that the right-hand side of inequality \eqref{V} is majorated by
$C_0\exp\!\left(-K n^{q\alpha}\right),$ 
where $C_0= 2 \exp\!\left(\frac{1}{q}-\frac{1}{2}\right)$,
and $K=\frac{1}{q}\left(\frac{\varepsilon}{c}\right)^q$.

We now show that the series
\[
\sum_{n=1}^{\infty} \Vh\left( \{ Z_n-\overline{m} > \varepsilon \} \cup \{ Z_n-\underline{m} < -\varepsilon \} \right) 
\]
is convergent. 
It suffices to prove the convergence of
\[
\sum_{n=1}^{\infty}\exp\!\left(-K n^{\beta}\right),
\qquad {\text{where}} \ \ \beta=q\alpha>0.
\]
The function $f(x)=\exp(-Kx^\beta)$ is positive and eventually decreasing.
By the integral test,
\[
\sum_{n=1}^{\infty}\exp(-K n^\beta)
\le
\int_{0}^{\infty}\exp(-Kx^\beta)\,dx.
\]
Performing the substitution $t=Kx^\beta$, we obtain
\[
\int_{0}^{\infty}\exp(-Kx^\beta)\,dx
=
\frac{1}{\beta}K^{-1/\beta}
\int_{0}^{\infty} t^{1/\beta-1}e^{-t}\,dt
=
\frac{1}{\beta}K^{-1/\beta}
\Gamma\!\left(\frac{1}{\beta}\right)
<\infty.
\]
Now, by the Borel-Cantelli lemma, it follows that the $\Vh$ measure is zero that infinitely many of the events
$$
\left( \{ Z_n-\overline{m} > \varepsilon \} \cup \{ Z_n-\underline{m} < -\varepsilon \} \right)
$$
occur.
As it is true for any $\varepsilon>0$, so \eqref{upper} is true.
\end{proof}

%%%%%%%%%%%%%%%%%%%%%%%%%%%%%%%%%%%%%%%%%%%%
\section{Strong law of large numbers for independent sub-Gaussian variables}
\label{Chap_slln}
Our aim is to obtain a strong law of large numbers for independent identically distributed sub-Gaussian random variables.
However, our Theorem \ref{iidSLLN} is true for more general situation.
In Example \ref{SLLNex}, we construct a plausible model fitting to Theorem \ref{iidSLLN}.

We shall use the independence notion given in \cite{ChenZ2013}.
The sequence of random variables $\xi_1, \xi_2, \dots$
is called independent if for each $n=1,2,\dots$ and each non-negative measurable functions $f_1, f_2, \dots$
we have 
$$
\Eh(f_1(\xi_1) f_2(\xi_2) \cdots f_n(\xi_n))  = \Eh(f_1(\xi_1)) \Eh(f_2(\xi_2)) \cdots \Eh(f_n(\xi_n)) .
$$
If $\xi_1, \xi_2, \dots$ are independent, then they satisfy the following property
\begin{eqnarray} \label{subIndep}
\Eh \prod_{i=1}^k \exp\left( {\lambda}  (\xi_i- m) \right) 
&\le&
\prod_{i=1}^k \Eh \exp\left( {\lambda}  (\xi_i- m) \right)  \nonumber \\ 
& &\text{for any real} \ \lambda, \ m , \ \text{and positive integer} \ k .
\end{eqnarray}
Now, let $\xi_1, \xi_2, \dots $ be random variables satisfying the sub-Gaussian property
(that is they are $\varphi_2$-sub-Gaussian):
for fixed constants
$\sigma > 0$, $\overline{m}$, and $\underline{m}$
\begin{equation}   \label{upGa}
\Eh (e^{\lambda(\xi_i - \overline{m}})  \le e^{(\sigma^2\lambda^2/2)}, \quad {\text {for}} \quad \lambda>0, \quad \text{and} \quad\Eh (e^{\lambda(\xi_i - \underline{m}}) \le e^{(\sigma^2\lambda^2/2)}, \quad {\text {for}} \quad \lambda<0.
\end{equation}
%and 
%\begin{equation}   \label{lowGa}
%\Eh (e^{\lambda(\xi_i - \underline{m}}) \le e^{(\sigma^2\lambda^2/2)}, \quad {\text {for}} %\quad \lambda<0.
%\end{equation}
Now, let $S_n = \xi_1 + \dots + \xi_n$ for any positive integer $n$.
We are ready to prove a strong law of large numbers. 
%%%%%%%%%%%%
%%
\begin{theorem}  \label{iidSLLN}
Let $\xi_1, \xi_2, \dots $ be random variables satisfying \eqref{subIndep} and \eqref{upGa} 
with $S_n$ defined above. Then
\begin{equation}
\label{eq:SLLN-capacity}
\Vh\Big(
\Big\{\liminf_{n\to\infty}\frac{S_n}{n}< \underline{m}\Big\}
\cup
\Big\{\limsup_{n\to\infty}\frac{S_n}{n}> \overline{m}\Big\}
\Big)=0.
\end{equation}
\end{theorem}
\begin{proof}
We shall apply Theorem \ref{MainTheorem} with $Z_n = \frac{S_n}{n}$.
Using the the sub-Gaussian property of $\xi_i$ and \eqref{subIndep},
\begin{eqnarray}
\Eh \exp( \lambda(Z_n - \overline{m}) ) =
\Eh \exp\left( \frac{\lambda}{n} \sum_{i=1}^n (\xi_i- \overline{m}) \right)=
\Eh \prod_{i=1}^n \exp\left( \frac{\lambda}{n}  (\xi_i- \overline{m}) \right)   \nonumber \\
\le
\prod_{i=1}^n \Eh\exp\left( \frac{\lambda}{n}  (\xi_i- \overline{m}) \right)
\le 
\prod_{i=1}^n \exp\left( \left(\frac{\lambda}{n}\right)^2  \left(\frac{\sigma^2}{2}\right) \right)
= \exp\left( \frac{\lambda^2}{n}  \left(\frac{\sigma^2}{2}\right) \right)
\end{eqnarray}
for $\lambda>0$.
Similarly,
\begin{eqnarray}
\Eh \exp( \lambda(Z_n - \underline{m}) ) =
\Eh \prod_{i=1}^n \exp\left( \frac{\lambda}{n}  (\xi_i- \underline{m}) \right)   \nonumber \\
\le
\prod_{i=1}^n \Eh\exp\left( \frac{\lambda}{n}  (\xi_i- \underline{m}) \right)
\le 
\prod_{i=1}^n \exp\left( \left(\frac{\lambda}{n}\right)^2  \left(\frac{\sigma^2}{2}\right) \right)
= \exp\left( \frac{\lambda^2}{n}  \left(\frac{\sigma^2}{2}\right) \right)
\end{eqnarray}
for $\lambda<0$.
So in Theorem \ref{MainTheorem},
$\tau_{\varphi_2}\!( Z_n )\le \,\sigma n^{-(1/2)}$ and it implies the result.
\end{proof}
%%%%%%%%%%%%%%%%%%%%%%%%
We see, that Theorem \ref{iidSLLN} is valid for independent identically distributed sub-Gaussian random variables.

In the following example, we shall use the well-known fact that for a random variable $X$ having normal distribution
with expectation $m$ and variance $\sigma^2>0$
$$
\E e^{\lambda X} = e^{\frac{\lambda^2 \sigma^2}{2} + \lambda m} .
$$

\begin{example} \label{SLLNex}
Let $\Omega$ be the real line and let $\FD$ be the family of its Borel sets.
Let $M$ be an arbitrary non-empty bounded set of real numbers,
let $\underline{m}=\inf(M)$ and $\overline{m}=\sup(M)$.
Let $\sigma>0$ be fixed.
Let $P_m$ denote the normal distribution with expectation $m$ and variance $\sigma^2$.
Then $(\Omega, \FD, P_m)$ is a usual probability space for any $m\in M$.
Let $\E_m X = \int_\Omega X dP_m$ be the usual expectation of the random variable $X$.
Let $\xi$ be the identity map: $\xi(\omega) = \omega$ for any $\omega\in \Omega$ (here $\Omega$ is the real line).
Then $\xi$ has a normal distribution with mean $m$ and variance $\sigma^2$.

Now, define the upper expectation as
$\Eh X = \sup \{ \E_m X \, : \, m\in M \}$
and the upper probability as $\Vh (A) = \sup \{ P_m (A) \, : \, m\in M \}$.
Then $\Eh \xi = \overline{m}$ and for the lower expectation
$\ED \xi = \underline{m}$.

For $\lambda>0$, from equation  
$$
\E_m e^{\lambda \xi- \overline{m}} = e^{\frac{\lambda^2 \sigma^2}{2} + \lambda (m-\overline{m})} 
$$
we see that $\Eh e^{\lambda \xi- \overline{m}} = e^{\frac{\lambda^2 \sigma^2}{2}}$,
so condition  \eqref{upGa} is satisfied, we have equality there, and $\overline{m}$ is the optimal constant in that 
condition.
Similarly,
for $\lambda<0$, 
we can see that $\Eh e^{\lambda X- \underline{m}} = e^{\frac{\lambda^2 \sigma^2}{2}}$
so condition  \eqref{upGa} is satisfied, we have equality there, and $\underline{m}$ is the optimal constant in that 
condition.

If $\xi_1, \xi_2, \dots $ are independent random variables, each of them has the same distribution as $\xi$, 
$S_n= \xi_1 + \cdots + \xi_n$, then the strong law of large numbers is satisfied, i.e.,
\eqref{eq:SLLN-capacity} holds.

However, we prefer the following explicit construction of the sequence $\xi_1, \xi_2, \dots $, for which 
\eqref{subIndep} is satisfied. We shall use the original idea of \cite{Huber1973} to construct negatively dependent random variables.
Consider copies $(\Omega^{(i)}, \FD^{(i)}, P_m^{(i)})$, $i=1,2,\dots$ of the probability space $(\Omega, \FD, P_m)$.
Then construct their product probability space 
$(\Omega^{(\infty)}, \FD^{(\infty)}, P_m^{(\infty)}) = \prod_{i=1}^\infty(\Omega^{(i)}, \FD^{(i)}, P_m^{(i)})$.
Then let 
$\xi_i$ be the $i$th coordinate random variable, i.e.,
$\xi_i(x_1, x_2, \dots ) = x_i$ for any $i$.
Then, under $P_m^{(\infty)}$, the random variables $\xi_1, \xi_2, \dots $ are independent and identically distributed,
each having distribution $P_m$, that is, a normal distribution with expectation $m$ and variance $\sigma^2$.
But we need upper probability and upper expectation.
So let $\Vh^{(\infty)}(A) =\sup_m P_m^{(\infty)}(A)$ for any event $A$ and
$\Eh^{(\infty)}(X) =\sup_m E_m^{(\infty)}(X)$ for an appropriate random variable $X$.
Let $f_1, f_2, \dots ,f_n$ be non-negative measurable functions.
Then
\begin{eqnarray*}
\Eh^{(\infty)}(f_1(\xi_1) f_2(\xi_2)) \cdots f_n(\xi_n)) =
\sup_m E_m^{(\infty)} (f_1(\xi_1) f_2(\xi_2)) \cdots f_n(\xi_n)) \\
= \sup_m E_m(f_1(\xi_1)) E_m (f_2(\xi_2)) \cdots E_m (f_n(\xi_n))  \\
\le
\sup_m E_m(f_1(\xi_1)) \cdot \sup_m E_m (f_2(\xi_2)) \cdots \sup_m E_m (f_n(\xi_n))\\
=
\sup_m E_m^{(\infty)}(f_1(\xi_1)) \cdot \sup_m E_m^{(\infty)} (f_2(\xi_2)) \cdots \sup_m E_m^{(\infty)} (f_n(\xi_n)) \\
=
\Eh^{(\infty)}(f_1(\xi_1)) \cdot \Eh^{(\infty)}( f_2(\xi_2)) \cdots \Eh^{(\infty)} (f_n(\xi_n)) .
\end{eqnarray*}
So \eqref{subIndep} is satisfied.
By Theorem \ref{iidSLLN}, the relation \eqref{eq:SLLN-capacity} is true, so the strong law is satisfied.
\end{example}

%\printcredits

\printbibliography
\bigskip
\noindent Nyanga Honda Masasila \\
University of Debrecen \\
Doctoral School of Informatics \\
Debrecen, Hungary \\
E-mail: \texttt{hondanyanga@gmail.com} \\

\bigskip
\noindent Istv\'an Fazekas \\
University of Debrecen \\
Faculty of Informatics\\
Debrecen, Hungary \\
E-mail: \texttt{fazekas.istvan@inf.unideb.hu} \\
\bigskip

\end{document}